\documentclass{amsart}
\calclayout
\usepackage{amssymb,amsmath,amsfonts,epsfig,latexsym,tikz}
\usepackage{tikz-cd}
\usepackage{hyperref}
\usepackage{enumerate}
\usepackage{mathtools}
\usepackage{verbatim}
\usepackage{cleveref}
\usepackage[shortlabels]{enumitem}
\usepackage{subcaption}
\usepackage{stmaryrd}

\usetikzlibrary{positioning}
\usetikzlibrary{matrix}
\usetikzlibrary{decorations}
\usetikzlibrary{decorations.pathreplacing, decorations.pathmorphing, angles,quotes}

\newtheorem{theorem}{Theorem}[section]
\newtheorem{definition}[theorem]{Definition}
\newtheorem{proposition}[theorem]{Proposition}
\newtheorem{lemma}[theorem]{Lemma}
\newtheorem{corollary}[theorem]{Corollary}

\theoremstyle{definition}
\newtheorem{remark}[theorem]{Remark}
\newtheorem{question}[theorem]{Question}
\newtheorem{example}[theorem]{Example}

\def\C{\mathbb{C}}

\def\N{\mathbb{N}}

\def\1{\mathbf{1}}

\def\<{\langle}
\def\>{\rangle}

\DeclareMathOperator{\codim}{codim}

\DeclareMathOperator{\Soc}{Soc}
\DeclareMathOperator{\im}{im}

\newcommand\nullset\varnothing

\begin{document}

\title{Complementary vectors of simplicial complexes}           

\author{Matt Larson and Alan Stapledon}

\begin{abstract}
We classify the complementary vectors of doubly Cohen--Macaulay complexes. This proves a conjecture of Swartz, negatively answers a question of Athanasiadis and Tzanaki, and gives new bounds on the number of independent sets in a matroid. Our technique works more generally for certain level quotients of Stanley--Reisner rings, giving new bounds on the face numbers of Buchsbaum* complexes. 
\end{abstract}

\maketitle

\vspace{-20 pt}

\setcounter{tocdepth}{1}
\tableofcontents

\section{Introduction}

We prove new bounds on the $h$-vectors of doubly Cohen--Macaulay simplicial complexes by calculating the Hilbert function of a certain quotient of the Stanley--Reisner ring of the complex. Our techniques can control the Hilbert function of certain quotients of the Stanley--Reisner rings of arbitrary simplicial complexes, and we work in that level of generality.

Let $\Delta$ be a simplicial complex of dimension $d-1$ with vertex set $V = \{ 1,\ldots, n \}$. Throughout the paper, we fix a field $k$ of characteristic $2$ and consider a purely transcendental field extension $K = k(a_{i,j})_{1 \le i \le d, \, 1 \le j \le n}$.
Let $K[\Delta]$ be the Stanley--Reisner ring of $\Delta$, and 
let  $\theta_i = \sum_{j} a_{i,j} x_j$ for $i \in \{1, \dotsc, d\}$. Then  $\theta_1, \dotsc, \theta_d$ is a linear system of parameters for $K[\Delta]$, and the corresponding quotient $H(\Delta) \coloneqq K[\Delta]/(\theta_1, \dotsc, \theta_d)$ is the 
\emph{generic artinian reduction} of $K[\Delta]$. 
Assume that $H^d(\Delta)$ is nonzero. Then
the \emph{level quotient} $\overline{H}(\Delta)$ of $H(\Delta)$ is the quotient of $H(\Delta)$ by the homogeneous ideal whose degree $q$ component is $\{ y \in H^q(\Delta) : y \cdot z = 0 \textrm{ for all } z  \in H^{d - q}(\Delta) \}$. By construction,  $\overline{H}(\Delta)$ is  a level algebra, i.e., for any nonzero $y \in \overline{H}^q(\Delta)$, there is $z \in \overline{H}^{d-q}(\Delta)$ with $y \cdot z$ nonzero. We are interested in bounding the following vector.

\begin{definition}
The $\overline{h}$-vector $(\overline{h}_0, \ldots, \overline{h}_d)$ of $\Delta$ is the Hilbert function of $\overline{H}(\Delta)$, i.e., $\overline{h}_q := \dim \overline{H}^q(\Delta)$ for $0 \le q \le d$.  
\end{definition}

We also consider the $\overline{g}$-vector 
$\overline{g}(\Delta) = (\overline{g}_0, \ldots, \overline{g}_{\lfloor d/2 \rfloor})$, where $\overline{g}_0 = \overline{h}_0 = 1$ and $\overline{g}_q = \overline{h}_q - \overline{h}_{q - 1}$ for $0 < q \le d/2$, which is an alternate encoding of the first half of $\overline{h}$. 
The $\overline{h}$-vector of $\Delta$ is independent of the choice of characteristic $2$ field $k$, see Proposition~\ref{prop:kindep}, and it is difficult to compute in general.

Before stating our results,  we recall some interesting special cases. 
Let $h(\Delta) = (h_0,\ldots,h_d)$ be the usual $h$-vector of $\Delta$, which is an encoding of the number of faces of $\Delta$ with $q$ elements for each $q$. 
By \cite[Corollary 3.2]{StanleyHilbert}, $\Delta$ is Cohen--Macaulay (over $k$) if and only if $h(\Delta)$ is the Hilbert function of $H(\Delta)$. 
Recall that $\Delta$ is doubly Cohen--Macaulay (over $k$) if it  is Cohen--Macaulay, and, for every vertex $v$, the deletion $\Delta \smallsetminus \{ v \}$ is Cohen--Macaulay of dimension $d-1$. Easy examples of doubly Cohen--Macaulay complexes, which were introduced by Baclawski \cite{BaclawskiCohenMacaulayconnectivity}, include triangulations of spheres and $2$-connected graphs. Being doubly Cohen--Macaulay is a topological property, i.e., whether $\Delta$ is doubly Cohen--Macaulay (over $k$) only depends on the topological space of the geometric realization of $\Delta$ \cite[Proposition III.3.7]{StanleyCombinatoricsCommutative}. 
If $h_d \not=0$ and $\Delta$ is Cohen--Macaulay, then $\Delta$ is doubly Cohen--Macaulay if and only if $H(\Delta)$ is level, i.e., $H(\Delta) = \overline{H}(\Delta)$  \cite[p. 94]{StanleyCombinatoricsCommutative}, which holds if and only if $\overline{h}(\Delta) = h(\Delta)$. 

There are many examples of doubly Cohen--Macaulay complexes coming from complexes with a \emph{convex ear decomposition} in the sense of Chari \cite[Section~3.2]{ChariTwoDecompositions}. Explicitly, $\Delta$ admits a convex ear decomposition if it contains a sequence of subcomplexes $\Sigma_1,\ldots,\Sigma_m$, where each $\Sigma_p$ is the boundary of a $d$-dimensional simplicial polytope,  $\Delta = \cup_p \Sigma_p$, and $\Sigma_q \cap (\cup_{p < q} \Sigma_p)$ is a $(d- 1)$-dimensional ball for all $1 < q \le m$ (see \cite[Remark~2.3]{AFConvexEar}). Swartz proved that if $\Delta$ admits a convex ear decomposition, then it is doubly Cohen--Macaulay 
\cite[Theorem~4.1]{SwartzgElements}. He observed that the proof works 
if one replaces each $\Sigma_j$ above with a $(d - 1)$-dimensional $k$-homology sphere and requires each $\Sigma_q \cap (\cup_{p < q} \Sigma_p)$ to be a $(d - 1)$-dimensional $k$-homology ball. Examples of simplicial complexes admitting convex ear decompositions (and hence of doubly Cohen--Macaulay complexes) include the independence complex of a coloop-free matroid
\cite[Theorems~2-3]{ChariTwoDecompositions}, 
the Bergman complex of a matroid \cite[Section~4]{NSInequalities}, and the augmented Bergman complex of a matroid \cite[Theorem~1.2]{AFConvexEar}.
See \cite{SwartzgElements,Schweig1,Schweig2} for further examples. 

More generally, Athanasiadis and Welker introduced the notion of a Buchsbaum* complex (over $k$) in \cite[Definition~1.2]{AWBuchsbaumComplexes}. 
All oriented 
simplicial  $k$-homology manifolds are Buchsbaum* \cite[Proposition~2.7]{AWBuchsbaumComplexes},
the Buchsbaum* complexes that are Cohen--Macaulay are precisely the doubly Cohen--Macaulay complexes \cite[Proposition~2.5]{AWBuchsbaumComplexes},  and there is  a generalization of convex ear decompositions that constructs Buchsbaum* complexes \cite[Theorem~3.10]{AWBuchsbaumComplexes}.
Generalizing a result of Novik and Swartz for oriented simplicial $k$-homology manifolds \cite{NovikSwartzGorensteinRings}, 
Nagel \cite[Corollary~4.1]{NagelLevelAlgebras} proved that if $\Delta$ is a 
Buchsbaum* complex, then 
\[
\overline{h}_q = \begin{cases}
h_q + \binom{d}{q} \sum_{p = 0}^{q - 1} (-1)^{q - p} \beta_{p} &\textrm{ for } q < d, \\
\beta_{d - 1} &\textrm{ for } q = d, \\
\end{cases}
\]
where $\beta_p = \dim \tilde{H}_p(|\Delta|;k)$ is
the $p$th reduced Betti number of the geometric realization $|\Delta|$ of $\Delta$ over $k$. 

There is a natural identification of $H^d(\Delta) = \overline{H}^d(\Delta)$ with the reduced cohomology group $\tilde{H}^{d-1}(|\Delta|, K)$, 
so linear maps $\mu \colon \overline{H}^d(\Delta) \to K$ can be identified with the reduced homology group $\tilde{H}_{d-1}(|\Delta|, K)$. See Section~\ref{sec:background} for details. Let $k'$ be a subfield of $K$. By the universal coefficient theorem, the natural map from $\tilde{H}_{d-1}(|\Delta|, k') \otimes_{k'} K$ to $\tilde{H}_{d-1}(|\Delta|, K)$ is an isomorphism. We say that a map $\mu \colon \overline{H}^d(\Delta) \to K$ is \emph{defined over $k'$} if it can be expressed as $\tilde{\mu} \otimes 1$, where $\tilde{\mu}$ corresponds to an element of $\tilde{H}_{d-1}(|\Delta|, k')$.   

For each nonzero map $\mu \colon \overline{H}^d(\Delta) \to K$, let $\overline{H}(\Delta, \mu)$ be the \emph{Gorenstein quotient} of $\overline{H}(\Delta)$, i.e., the quotient of $\overline{H}(\Delta)$ by the homogeneous ideal whose degree $q$ component is $\{y \in \overline{H}^q(\Delta) : \mu(y \cdot z) = 0 \text{ for all }z \in \overline{H}^{d-q}(\Delta)\}$. By construction, $\overline{H}(\Delta, \mu)$ is equipped with an isomorphism $\mu \colon \overline{H}^{d}(\Delta, \mu) \to K$, and the pairing $\overline{H}^q(\Delta, \mu) \times \overline{H}^{d-q}(\Delta, \mu) \to K$ given by $(y, z) \mapsto \mu(y \cdot z)$ is nondegenerate. 
The following result, which follows easily from \cite[Theorem I]{APP} (see also Proposition~\ref{prop:anistropycycle}), significantly restricts the possible $\overline{h}$-vectors of simplicial complexes with nonzero $\beta_{d-1}$. 

\begin{theorem}\label{thm:SLP}
Let $\Delta$ be a simplicial complex of dimension $d-1$, and let $\mu \colon \overline{H}^d(\Delta) \to K$ be a nonzero map which is defined over $\mathbb{F}_2$. Let $\ell = x_1 + \dotsb + x_n$. For each $q \le d/2$, multiplication by $\ell^{d - 2q}$ induces an isomorphism from $\overline{H}^q(\Delta, \mu)$ to $\overline{H}^{d-q}(\Delta, \mu)$. 
\end{theorem}

Choosing a basis $\mu_1, \dotsc, \mu_s$ for the dual of $\overline{H}^d(\Delta)$, we have an injective ring homomorphism
$$\overline{H}(\Delta) \hookrightarrow \bigoplus_{p=1}^{s} \overline{H}(\Delta, \mu_p).$$
We can choose each $\mu_p$ to be defined over $\mathbb{F}_2$, so multiplication by $\ell^{d-2q}$ induces an isomorphism from the degree $q$ part of the right-hand-side to the degree $d - q$ part of the right-hand-side, giving the following corollary. 

\begin{corollary}\label{cor:gelement}\cite[Corollary 3.1]{APP}
Let $\Delta$ be a simplicial complex of dimension $d-1$ with $\beta_{d - 1} \neq 0$. Let $\ell = x_1 + \dotsb + x_n$. For each $q \le d/2$, multiplication by $\ell^{d - 2q}$ induces an injection from $\overline{H}^q(\Delta)$ to $\overline{H}^{d-q}(\Delta)$. 
\end{corollary}

An $M$-vector (or $O$-sequence) is the Hilbert function of a finite dimensional standard graded $K$-algebra, i.e., a finite dimensional graded $K$-algebra $H$ which has $H^0 = K$ and which is generated by $H^1$. Macaulay gave explicit inequalities classifying $M$-vectors, see Proposition~\ref{prop:Mvec}. 
As $\overline{H}(\Delta)/(\ell)$ is generated in degree $1$ and its Hilbert function begins $(\overline{h}_0, \overline{h}_1 - \overline{h}_0, \dotsc, \overline{h}_{\lfloor d/2 \rfloor} - \overline{h}_{\lfloor d/2 \rfloor - 1}, \dotsc)$, we have the following corollary.

\begin{corollary}\label{cor:topheavy}
Let $\Delta$ be a simplicial complex of dimension $d-1$ with 
$\beta_{d - 1} \neq 0$. 
Then
$\overline{h}(\Delta)$ satisfies the following 
properties:
\begin{enumerate}
\item\label{i:ineq1} $\overline{h}_q \le \overline{h}_{d - q}$ for each $q \le d/2$,
\item\label{i:ineq2} $\overline{g}(\Delta)$ is an $M$-vector. 
\end{enumerate}
\end{corollary}

Moreover, for every $M$-vector $a = (a_0,\ldots,a_{\lfloor d/2 \rfloor})$, there is a $(d - 1)$-dimensional simplicial complex $\Delta$ that is the boundary of a $d$-dimensional simplicial polytope with $a = \overline{g}(\Delta)$ \cite{BLProofSufficiency}, so Corollary~\ref{cor:topheavy} is a complete classification of the $\overline{g}$-vectors of simplicial complexes with nonvanishing $\beta_{d-1}$. 

Properties \eqref{i:ineq1} and \eqref{i:ineq2} in Corollary~\ref{cor:topheavy} were proved by Swartz in the special case when 
$\Delta$ admits a convex ear decomposition \cite[Theorem~3.9]{SwartzgElements}, 
and they were conjectured by Bj\"orner and Swartz when $\Delta$ is doubly Cohen--Macaulay \cite[Problem~4.2]{SwartzgElements} and by Athanasiadis and Welker when $\Delta$ is Buchsbaum* \cite[Question~5.7]{AWBuchsbaumComplexes}. 

\medskip

We prove a strengthening of Corollary~\ref{cor:gelement} which gives further numerical consequences for $\overline{h}$-vectors. In light of the classification of possible $\overline{g}$-vectors in Corollary~\ref{cor:topheavy}, it is natural to consider the following vector.
\begin{definition}
The \emph{complementary vector} of $\Delta$ is $\overline{c}(\Delta) = (\overline{c}_0, \ldots, \overline{c}_{\lfloor (d - 1)/2 \rfloor})$, where $\overline{c}_q = \overline{h}_{d - q} - \overline{h}_{q}$ for each $q \le (d - 1)/2$. 
\end{definition}

When $\Delta$ is doubly Cohen--Macaulay, $\overline{c}_q = h_{d - q} - h_q$ for all $q$, and the vector $\overline{c}(\Delta)$ was studied by Swartz \cite{SwartzgElements}, who called it the \emph{complementary $h$-vector} of $\Delta$. We will classify the possible complementary vectors of simplicial complexes using the following algebraic result.

\begin{theorem}\label{thm:gorensteinhilbert}
Let $\Delta$ be a simplicial complex of dimension $d-1$. Let $k$ be a field of characteristic $2$, and assume that $\beta_{d - 1} \neq 0$. 
Then there is a map $\mu \colon \overline{H}^d(\Delta) \to K$ such that, for each $q \le d/2$, $\dim \overline{H}^q(\Delta, \mu) = \overline{h}_q$ and multiplication by $\ell^{d-2q}$ induces an isomorphism from $\overline{H}^q(\Delta, \mu)$ to $\overline{H}^{d-q}(\Delta, \mu)$. If $k$ is infinite, then $\mu$ can be chosen to be defined over $k$.  
\end{theorem}

This is indeed a strengthening of Corollary~\ref{cor:gelement}, as the isomorphism from $\overline{H}^q(\Delta) = \overline{H}^q(\Delta, \mu)$ to $\overline{H}^{d-q}(\Delta, \mu)$ given by multiplication by $\ell^{d - 2q}$ factors through $\overline{H}^{d-q}(\Delta)$.

If $\Delta$ admits a convex ear decomposition, then 
Swartz proved that $\overline{c}(\Delta)$ 
is a sum of $M$-vectors \cite[Proposition~4.4]{SwartzgElements}, and he asked if the same result holds more generally when $\Delta$ is doubly Cohen--Macaulay \cite[Problem~4.8]{SwartzgElements}. Using Theorem~\ref{thm:gorensteinhilbert}, we give the following classification of the possible complementary vectors of simplicial complexes, affirmatively answering Swartz's question. 

\begin{corollary}\label{cor:complementaryvector}
Let $\Delta$ be a simplicial complex of dimension $d-1$. 
 Let $k$ be a field of characteristic $2$, and assume that $\beta_{d - 1} \neq 0$. 
Then the complementary vector $\overline{c}(\Delta)$ is a sum of $M$-vectors. 
Moreover, suppose that a sequence 
$a = (a_0,\ldots,a_{\lfloor (d - 1)/2 \rfloor})$ is a sum of $M$-vectors. 
Then there is a $(d - 1)$-dimensional simplicial complex $\Delta$ that admits a convex ear decomposition  such that $a = \overline{c}(\Delta)$.
\end{corollary}

A sequence of nonnegative integers $(a_0, a_1, \dotsc, a_m)$ is a sum of $M$-vectors if and only if there is a graded module $M = M^0 \oplus M^1 \oplus \dotsb$ over a polynomial ring in finitely many variables which is generated as a module in degree $0$ and has $\dim M^q = a_q$ for all $q$. We show that $(a_0, a_1, \dotsc, a_m)$ is a sum of $M$-vectors if and only if it is identically $0$, or if $a_0 > 0$ and $(1, a_1, \dotsc, a_m)$ is an $M$-vector. See Proposition~\ref{prop:summvec}. In particular, Macaulay's inequalities classifying $M$-vectors immediately imply a numerical classification of all sums of $M$-vectors. 

\begin{proof}[Proof of Corollary~\ref{cor:complementaryvector}]
By Theorem~\ref{thm:gorensteinhilbert}, there is a linear map
$\mu \colon \overline{H}^d(\Delta) \to K$ 
such that $\dim \overline{H}^q(\Delta, \mu) = \overline{h}_q$ for each $q \le d/2$. 
Let $M$ be the kernel of the quotient map $\overline{H}(\Delta) \to \overline{H}(\Delta, \mu)$.
We have $\dim M^q =  \overline{h}_{q} - \dim \overline{H}^q(\Delta, \mu) = \overline{h}_{q} - \dim \overline{H}^{d - q}(\Delta, \mu)$ for all $q$. In particular, our assumption implies that $\dim M^q = 0$ for $q \le d/2$ and that $\dim M^q = \overline{h}_{q} - \overline{h}_{d - q}$ for $q > d/2$. 
Let $M^*$ be the graded dual of $M$, i.e., the degree $q$ component of $M^*$ is the dual of the vector space $M^{d-q}$.  This is a module over $\overline{H}(\Delta)$, with $y \in \overline{H}(\Delta)$ acting on $f \in M^*$ via $(y \cdot f)(z) = f(y \cdot z)$ for each $z \in M$. See, for example, \cite[Section~21.1]{EisenbudCommutativeAlgebra}.
The Hilbert function of $M^*$ is $\overline{c}(\Delta)$.

It 
suffices to show that $M^*$ is generated in degree $0$ as a module over $\overline{H}(\Delta)$, as $\overline{H}(\Delta)$ is generated in degree $1$ and so is a quotient of a polynomial ring in finitely many variables. The perfect pairing between $M$ and $M^*$ induces a perfect pairing between 
$M^*/\overline{H}^{>0} \cdot M^*$ 
and $\Soc(M) := \{y \in M : y \cdot z = 0 \text{ for all }z \in \overline{H}^{>0}(\Delta)\}$. As $M$ is an ideal in the level algebra $\overline{H}(\Delta)$, 
$\Soc(M)$ is concentrated in degree $d$, so 
$M^*/\overline{H}^{>0} \cdot M^*$ 
is concentrated in degree $0$. Then  $M^*$ is generated in degree $0$ by the graded version of Nakayama's lemma. 

The construction of all sums of $M$-vectors as complementary vectors of simplicial complexes with convex ear decompositions is carried out in Proposition~\ref{prop:construction}.
\end{proof}

Using this classification of complementary vectors, we give a negative answer to \cite[Question~7.2]{ATSymmetricDecompositions}, see Example~\ref{ex:AT}.

Corollary~\ref{cor:topheavy} and Corollary~\ref{cor:complementaryvector} are still far from classifying the possible $\overline{h}$-vectors of simplicial complexes 
as they do not account for the interactions between $\overline{g}(\Delta)$ and $\overline{c}(\Delta)$. We have the following further result in this direction.
Our proof of Corollary~\ref{cor:complementaryvector} shows that $\overline{c}(\Delta)$ is the Hilbert function of a graded module over the standard graded algebra $\overline{H}(\Delta)$ that is generated in degree $0$. A generalization of Macaulay's inequalities
for $M$-vectors to modules proved in \cite{HulGeneralizationMacaulay,BlancafortElias} then provides additional inequalities satisfied by $\overline{c}(\Delta)$ that depend on $\overline{h}_1$, see Corollary~\ref{cor:modulebound}. These inequalities are new even when $\Delta$ admits a convex ear decomposition. For example, they provide new inequalities for the $h$-vector of the independence complex of a coloop-free matroid.

\medskip 

The Hilbert functions of standard graded level algebras, such as $\overline{h}$-vectors of simplicial complexes, have been extensively studied. A full collection of references is too vast to list; for a survey of results up to 2003 see \cite{GHMSHilbertFunctionLevelAlgebra}, and for some newer results see \cite{ASNonUnimodal,MTStructureInverse, PSZGorensteinInterval, ZanelloLogConcavity, IarrobinoLogConcave, BGIZMinimalGorenstein}. 
In that more general setting, the conclusions of Corollary~\ref{cor:topheavy} and Corollary~\ref{cor:complementaryvector} can fail:  if $(a_0, \dotsc, a_d)$ is the Hilbert function of a level algebra, then the complementary vector $(a_d - a_0, a_{d-1} -a_1, \dotsc)$ can have negative entries; for example, $(1,13,12,13)$ is the Hilbert function of a level algebra by \cite[Example~4.3]{StanleyHilbertGradedAlgebras}. Even the Hilbert function of a level algebra which satisfies the conclusion of Corollary~\ref{cor:topheavy} need not satisfy the conclusion of Corollary~\ref{cor:complementaryvector}, see Example~\ref{ex:zanello}. A level algebra may satisfy the conclusion of Corollary~\ref{cor:gelement}, but not the conclusion of Theorem~\ref{thm:gorensteinhilbert}, see Example~\ref{ex:monomial}.

The closest generalization of the results of this paper to Hilbert functions of arbitrary standard graded level algebras is \cite[Corollary~2.11]{GHMSHilbertFunctionLevelAlgebra}, which states that $(a_d, \dotsc, a_0)$ can be written as  the Hilbert function of a standard graded Gorenstein algebra plus a sum of $M$-vectors. Our deduction of Corollary~\ref{cor:complementaryvector} from Theorem~\ref{thm:gorensteinhilbert} is similar to the proof of \cite[Corollary~2.11]{GHMSHilbertFunctionLevelAlgebra}. 

On the other hand, we present generalizations of our results 
over  fields  of 
arbitrary characteristic
in Section~\ref{ss:positivechar}, and for balanced simplicial complexes in Section~\ref{ss:balanced}. We also prove an analogue of Corollary~\ref{cor:complementaryvector} for pure $O$-sequences (see Theorem~\ref{thm:pure}).

\medskip

Our proof of Theorem~\ref{thm:gorensteinhilbert} is based on the \emph{anisotropy} of the Hodge--Riemann forms on $\overline{H}(\Delta, \mu)$, see Proposition~\ref{prop:anistropycycle}. This follows from the differential operators technique, a technique that was introduced by Papadakis and Petrotou \cite{PapadakisPetrotoug} to give a proof of the algebraic $g$-conjecture for simplicial spheres, which is a special case of Theorem~\ref{thm:SLP}. The differential operators technique has subsequently been extended to more general simplicial complexes \cite{APP,KaruXiaoAnisotropy}.

\subsection*{Acknowledgments}
We thank Christos Athanasiadis, Luis Ferroni, Anthony  Iarrobino, Kalle Karu, Juan Migliore,  and Fabrizio Zanello for useful conversations. In particular, we thank Luis Ferroni for Example~\ref{ex:ultralog}, and we thank Fabrizio Zanello for Example~\ref{ex:zanello}. We thank the referee for their helpful comments. This paper was proofread using Claude Opus 4.8.

\section{Cycles and Lefschetz}\label{sec:background}

In this section, we establish some properties of $\overline{H}(\Delta)$. Although the results in this section follow from standard techniques, they are not available in the precise form that we need, so we sketch the proofs.
Throughout, we continue with the setup and notation of the introduction. In particular,
$\Delta$ is a simplicial complex of dimension $d-1$ with vertex set $V = \{ 1,\ldots, n \}$,
$k$ is a field of characteristic $2$, and $K = k(a_{i,j})_{1 \le i \le d, \, 1 \le j \le n}$. We continue to assume that  $H^d(\Delta)$ is nonzero.

\medskip

We first recall how to identify ${H}^d(\Delta) = \overline{H}^d(\Delta)$ with the reduced cohomology $\tilde{H}^{d-1}(|\Delta|, K)$, or, equivalently, how to identify the dual vector space $\operatorname{Hom}(H^d(\Delta), K)$ with $\tilde{H}_{d-1}(|\Delta|, K)$. 

Let $x_1, \dotsc, x_n \in H^1(\Delta)$ be the classes corresponding to vertices of $\Delta$. For a face $G$ of $\Delta$, let $x^G = \prod_{j \in G} x_j$. 
Each $H^s(\Delta)$ is spanned by $\{x^G : G \text{ face of }\Delta, \, |G| = s\}$ \cite[Theorem 9]{Lee}.

Suppose that $F_1, \dotsc, F_t$ are the facets of $\Delta$ of size $d$. For each facet $F = \{j_1, \dotsc, j_d\}$ of size $d$, let $[F]$ be the determinant of the $d \times d$ matrix whose $(i,p)$th entry is $a_{i,j_p}$. 
Using simplicial homology, we can identify $\tilde{H}_{d-1}(|\Delta|, K)$ with a subspace of $K^t$. Then each $y = (c_1, \dotsc, c_t) \in \tilde{H}_{d-1}(|\Delta|, K)$ defines a map 
$$\mu_y \colon K^t \to K, \quad \mu_y(e_p) = \frac{c_p}{[F_p]},$$
where $e_p$ is the $p$th standard basis vector.
As $H^{d}(\Delta)$ is spanned by monomials corresponding to facets of size $d$, $H^d(\Delta)$ can be identified with a quotient of $K^t$. 
The condition that $y$ lies in $\tilde{H}_{d-1}(|\Delta|, K)$ is exactly the condition that $\mu_y$ factors through $H^d(\Delta)$ \cite[Theorem 10]{Lee}. 
Therefore, a map $\mu \colon H^d(\Delta) \to K$ is defined over a subfield $k'$ of $K$ if and only if $\mu(x^F) \cdot [F]$ lies in $k'$ for each facet $F$ of size $d$.

We say that a symmetric bilinear form $B$ on a vector space $V$ is \emph{anisotropic} if, for all nonzero $x \in V$, $B(x,x)$ is nonzero. The following result is a strengthening of Theorem~\ref{thm:SLP}. It is a minor modification of \cite[Theorem III]{APP} or \cite[Theorem 4.4]{KaruXiaoAnisotropy}, and it is essentially the same as \cite[Theorem 7.6]{Oba}.

\begin{proposition}\label{prop:anistropycycle}
Let $\mu \colon \overline{H}^d(\Delta) \to K$ be a nonzero map which is defined over $\mathbb{F}_2$, and let $\overline{H}(\Delta, \mu)$ be the corresponding Gorenstein quotient. Let $\ell = x_1 + \dotsb + x_n \in \overline{H}^1(\Delta, \mu)$. For each $q \le d/2$, the symmetric bilinear form $\overline{H}^q(\Delta, \mu) \times \overline{H}^q(\Delta, \mu) \to K$ given by $(x, y) \mapsto \mu(x \cdot y \cdot \ell^{d - 2q})$ is anisotropic. 
\end{proposition}

\begin{proof}
If one can verify that the analogue of \cite[Theorem 4.1]{KaruXiaoAnisotropy} holds in $\overline{H}(\Delta, \mu)$, then one can prove the result in exactly the same way as \cite[Theorem 4.4]{KaruXiaoAnisotropy} is proved. As in \cite[Section 2.6]{KaruXiaoAnisotropy}, by introducing an additional cone vertex, the map from the degree $d$ part of $K[x_1, \dotsc, x_n]$ to $K$ induced by $\mu$ can be written as an $\mathbb{F}_2$-linear combination of maps which are boundaries of simplices. The argument in \cite[Lemma 4.5]{KaruXiaoAnisotropy} then reduces the desired statement to the case of the boundary of a simplex, which is treated in \cite[Theorem 4.6]{KaruXiaoAnisotropy}. 
\end{proof}

Finally, we prove that the $\overline{h}$-vector does not depend on the choice of field $k$ of characteristic $2$. 

\begin{proposition}\label{prop:kindep}
The dimension of $\overline{H}^q(\Delta)$ is independent of the choice of field $k$ of characteristic $2$.
\end{proposition}

\begin{proof}
Choose a basis $\mu_1, \dotsc, \mu_s$ for $\tilde{H}_{d-1}(|\Delta|, \mathbb{F}_2)$. As $\tilde{H}_{d-1}(|\Delta|, K)$ is identified with $\tilde{H}_{d-1}(|\Delta|, \mathbb{F}_2) \otimes_{\mathbb{F}_2} K$, this induces a basis for $\operatorname{Hom}(H^d(\Delta), K)$. The kernel of $H^q(\Delta) \to \overline{H}^q(\Delta)$ can be written as an intersection of subspaces indexed by some element $\mu_t$ in our chosen basis and a monomial $x^J$ in $K[\Delta]$ of degree $d - q$, where the subspace corresponding to $\mu_t$ and $x^J$ is the kernel of the map $H^q(\Delta) \to K$ given by $y \mapsto \mu_t(y \cdot x^J)$. As these subspaces are all defined over $\mathbb{F}_2(a_{i,j})_{1 \le i \le d, \, 1 \le j \le n}$, this implies that the dimension of $\overline{H}^q(\Delta)$ is the same as the dimension of the level quotient of the generic artinian reduction over $\mathbb{F}_2$. 
\end{proof}

\begin{question}
Let $\Delta$ be a simplicial complex of dimension $d-1$.  Throughout this paper, we assume that $\beta_{d - 1} = \dim H^d(\Delta)$ is nonzero. 
Can any of the results of this paper be extended to the case when $\beta_{d - 1} = 0$? In this case, if $m = \max \{ q : H^q(\Delta) \neq 0 \}$, then 
the level quotient $\overline{H}(\Delta)$ of $H(\Delta)$ is the quotient of $H(\Delta)$ by the homogeneous ideal whose degree $q$ component is $\{ y \in H^q(\Delta) : y \cdot z = 0 \textrm{ for all } z  \in H^{m - q}(\Delta) \}$. For example, if $\hat{\Delta}$ is the join of $\Delta$ with an $(r - 1)$-dimensional simplex, then 
$\dim \hat{\Delta} = d + r - 1$, $H(\hat{\Delta}) \cong H(\Delta)$, and 
$\overline{H}(\hat{\Delta}) \cong \overline{H}(\Delta)$ (after possibly extending the underlying field $K$). In particular, is there an analogue of Proposition~\ref{prop:anistropycycle}?
\end{question}

\section{Generic cycles}

In this section, we prove Theorem~\ref{thm:gorensteinhilbert}. We will deduce it as a formal algebraic consequence of Proposition~\ref{prop:anistropycycle}. For this, we will use Adiprasito's ``perturbation lemma.'' For self-containedness, we include a short and elementary proof. 

\begin{lemma}\label{lem:perturbation}\cite[Lemma~6.1]{Adiprasitog}
Let $V,W$ be finite dimensional vector spaces over a field $k$. Let $\alpha,\beta \colon V \to W$ be linear maps such that $\beta(\ker(\alpha)) \cap \im(\alpha) = \{ 0 \}$. Then, for all but finitely many $\lambda \in k$, $\ker(\alpha + \lambda \beta) = \ker(\alpha) \cap \ker(\beta)$. 
\end{lemma}
\begin{proof}
Let $s = \codim \ker(\alpha)$. Choose a splitting $V = A \oplus \ker(\alpha)$, so $A$ is some subspace with $\dim A = s$. Choose a splitting $W = \im(\alpha) \oplus B$, so $\dim B = \dim W - s$. By assumption, $\beta(\ker(\alpha)) \cap \im(\alpha) = \{ 0 \}$, and so we may assume that $\beta(\ker(\alpha)) \subset B$. 
After choosing bases for the above subspaces,  
the matrix corresponding to $\alpha$ has the form
\[
\begin{bmatrix}
M & 0 \\
0 & 0 \\
\end{bmatrix}
\]
for some invertible $s \times s$ matrix $M$. The matrix corresponding to $\beta$ has the form
\[
\begin{bmatrix}
N & 0 \\
N' & N'' \\
\end{bmatrix}
\]
for some $s \times s$ matrix $N$, some $(\dim W - s) \times s$ matrix $N'$,  and some $(\dim W - s) \times (\dim V - s)$ matrix $N''$.
Observe that, if $\lambda \not= 0$, then $M + \lambda N$ is invertible if and only if $\lambda^{-1}I + M^{-1} N$ is invertible. This happens if and only if  $-\lambda^{-1}$ is not one of the (finitely many) eigenvalues of  $M^{-1} N$. 
Finally, if $M + \lambda  N$ is invertible, then $\ker(\alpha + \lambda \beta) = 
0_A \oplus \ker(N'') 
= \ker(\alpha) \cap \ker(\beta)$, where $0_A$ denotes the origin in $A$. 
\end{proof}

We say that $\overline{H}(\Delta, \mu)$ has the \emph{strong Lefschetz property} if, for each $q \le d/2$, multiplication by $\ell^{d - 2q}$ induces an isomorphism from $\overline{H}^q(\Delta, \mu)$ to $\overline{H}^{d-q}(\Delta, \mu)$. For example, Theorem~\ref{thm:SLP} states that $\overline{H}(\Delta, \mu)$ has the strong Lefschetz property if $\mu$ is defined over $\mathbb{F}_2$. Using the argument in \cite[Lemma 5.1]{LNS}, it can be shown that this is equivalent to the usual definition of the strong Lefschetz property. 

\begin{proof}[Proof of Theorem~\ref{thm:gorensteinhilbert}]
Choose a basis $\mu_1, \dotsc, \mu_s$ of $\operatorname{Hom}(H^d(\Delta), K) = \tilde{H}_{d-1}(|\Delta|, K)$ which is defined over $\mathbb{F}_2$. We show by induction on $t$ that there are $(\lambda_1, \dotsc, \lambda_t) \in K^t$ such that, if we set $\mu = \sum_{p=1}^{t} \lambda_p \mu_p$, then for each $q \le d/2$, the kernel of $\overline{H}^q(\Delta) \to \overline{H}^q(\Delta, \mu)$ is the same as the kernel of $\overline{H}^q(\Delta) \to \oplus_{p=1}^{t} \overline{H}^q(\Delta, \mu_p)$ and $\overline{H}(\Delta, \mu)$ has the strong Lefschetz property. As $\overline{H}^q(\Delta)$ injects into $\oplus_{p=1}^{s} \overline{H}^q(\Delta, \mu_p)$, this implies the result. 

The case $t = 1$ holds by setting $\lambda_1 = 1$. 
For the inductive step, suppose that $t < s$ and, for each $q \le d/2$, $\mu = \sum_{p=1}^{t} \lambda_p \mu_p$ has the property that the kernel of $\overline{H}^q(\Delta) \to \overline{H}^q(\Delta, \mu)$ is the same as the kernel of $\overline{H}^q(\Delta) \to \oplus_{p=1}^{t} \overline{H}^q(\Delta, \mu_p)$ and $\overline{H}(\Delta, \mu)$ has the strong Lefschetz property. For $\nu \in \operatorname{Hom}(\overline{H}^d(\Delta), K)$ and some $q \le d/2$, let $\phi_{\nu} \colon \overline{H}^q(\Delta) \to \operatorname{Hom}(\overline{H}^q(\Delta), K)$ be the map given by $\phi_{\nu}(x) = (y \mapsto \nu(x \cdot y \cdot \ell^{d - 2q}))$. The map which sends $\nu \in \operatorname{Hom}(\overline{H}^d(\Delta), K)$ to $\phi_{\nu}$ is $K$-linear. 

Note that $\overline{H}(\Delta, \nu)$ has the strong Lefschetz property if and only if  
$\ker(\phi_{\nu}) = \{x \in \overline{H}^q(\Delta) : \nu(x \cdot y \cdot \ell^{d - 2q}) = 0 \text{ for all }y \in \overline{H}^q(\Delta)\}$
coincides with the kernel of the map $\overline{H}^q(\Delta) \to \overline{H}^q(\Delta, \nu)$. By the induction hypothesis, the kernel of $\phi_{\mu}$ coincides with the kernel of $\overline{H}^q(\Delta) \to \oplus_{p=1}^{t} \overline{H}^q(\Delta, \mu_p)$. 
We claim that
\begin{equation}\label{eq:phi}
\phi_{\mu_{t+1}}(\ker(\phi_{\mu})) \cap \im(\phi_{\mu}) = 0.
\end{equation}
Assuming this claim, we apply Lemma~\ref{lem:perturbation} with $\alpha = \phi_\mu$ and $\beta = \phi_{\mu_{t+1}}$ to deduce that, for all but finitely many $\lambda \in K$, we have
$$ \ker(\phi_{\mu_{t+1}}) \cap \ker(\phi_{\mu}) = \ker(\phi_{\mu} + \lambda \phi_{\mu_{t+1}}) =  \ker(\phi_{{\mu} + \lambda\mu_{t+1}}) .$$
We can find a $\lambda$ satisfying this condition because $K$ is infinite.
For any such $\lambda$, the kernel of $\phi_{\mu} + \lambda \phi_{\mu_{t+1}}$ is the same as the kernel of $\overline{H}^q(\Delta) \to \oplus_{p=1}^{t+1} \overline{H}^q(\Delta, \mu_p)$. This implies that the inclusions
$$\ker(\overline{H}^q(\Delta) \to \oplus_{p=1}^{t+1} \overline{H}^q(\Delta, \mu_{p})) \subseteq \ker(\overline{H}^q(\Delta) \to \overline{H}^q(\Delta, \mu + \lambda \mu_{t+1})) \subseteq \ker(\phi_{\mu} + \lambda \phi_{\mu_{t+1}})$$
must be equalities, so $\overline{H}(\Delta, \mu + \lambda \mu_{t+1})$ has the strong Lefschetz property. 

It remains to prove \eqref{eq:phi}. Let $x$ be a class in $\ker(\phi_{\mu})$ such that there is some $y \in \overline{H}^q(\Delta)$ with $\phi_{\mu_{t+1}}(x) = \phi_{\mu}(y)$. We need to show that $\phi_{\mu_{t+1}}(x) = 0$. By definition, for any $z \in \overline{H}^q(\Delta)$, we have
$$\mu_{t+1}(x \cdot z \cdot \ell^{d - 2q}) = \mu(y \cdot z \cdot \ell^{d - 2q}).$$
Set $z = x$ in the above equation. As $x \in \ker(\phi_{\mu})$, we have $\mu(y \cdot x \cdot \ell^{d - 2q}) = 0$. Therefore $\mu_{t+1}(x \cdot x \cdot \ell^{d - 2q}) = 0$. As $\mu_{t+1}$ is defined over $\mathbb{F}_2$, Proposition~\ref{prop:anistropycycle} then implies that $x$ lies in the kernel of $\overline{H}^q(\Delta) \to \overline{H}^q(\Delta, \mu_{t+1})$, which is the kernel of $\phi_{\mu_{t+1}}$. 

Finally, if $k$ is infinite, then at each step we can choose $\lambda_t$ to lie in $k$, so $\mu$ will be defined over $k$. 
\end{proof}

The proof of Theorem~\ref{thm:gorensteinhilbert} relies on Lemma~\ref{lem:perturbation}, whose hypothesis is difficult to verify with weaker assumptions. The following example shows that the conclusion of Theorem~\ref{thm:gorensteinhilbert} can fail if one uses a non-generic artinian reduction.

\begin{figure}
\begin{tikzpicture}[scale=4]
\draw (0,0) node[left] { $1$ } -- 
(1,0) node[right] { $2$ } -- 
(1,1) node[right] { $3$ } -- (0,1) node[left]{$4$} -- (0,0);
\draw (0,1) -- (0.5, 0.5) node[right] {$5$} -- (1,0);
\fill (0, 0) circle[radius = 0.5pt];
\fill (1, 0) circle[radius = 0.5pt];
\fill (0, 1) circle[radius = 0.5pt];
\fill (1, 1) circle[radius = 0.5pt];
\fill (0.5, 0.5) circle[radius = 0.5pt];
\end{tikzpicture}
\quad \quad \quad \quad 
\begin{tikzpicture}[scale=4]
\draw (0,0) node[left] { $1$ } -- 
(1,0) node[right] { $2$ } -- 
(1,1) node[right] { $3$ } -- (0,1) node[left]{$4$} -- (0,0);
\draw (0,1) -- (1,0);
\fill (0, 0) circle[radius = 0.5pt];
\fill (1, 0) circle[radius = 0.5pt];
\fill (0, 1) circle[radius = 0.5pt];
\fill (1, 1) circle[radius = 0.5pt];
\end{tikzpicture}
\caption{Some counterexamples to possible extensions of Theorem~\ref{thm:gorensteinhilbert}.}\label{fig:nongeneric}
\end{figure}
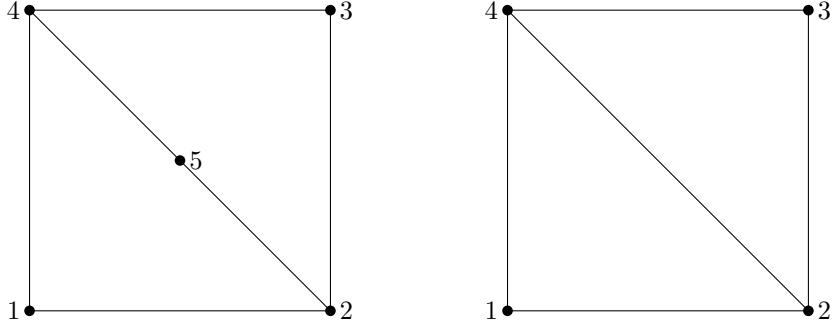

\begin{example}\label{ex:nongeneric}
Let $\Delta$ be the $1$-dimensional complex on the left of Figure~\ref{fig:nongeneric}, which is doubly Cohen--Macaulay and has $h$-vector $(1, 3, 2)$. 
Set $\theta_1^N = a_{1,1} x_1 + a_{1,3} x_3 + a_{1,5} x_5$ and $\theta_2 = a_{2,1} x_1 + a_{2,2} x_2 + a_{2,3} x_3 + a_{2,4} x_4 + a_{2,5} x_5$. Let $H_N(\Delta) = K[\Delta]/(\theta_1^N, \theta_2)$, which is level and has Hilbert function $(1, 3, 2)$. A basis of $\tilde{H}_1(|\Delta|, K)$ is given by the classes of the induced subcomplexes with vertex sets $\{1, 2, 4, 5\}$ and $\{2, 3, 4, 5\}$, respectively. Call these cycles $\mu_1$ and $\mu_2$. Let $\mu = \alpha \mu_1 + \beta \mu_2$ be a nonzero cycle. If $\alpha$ and  $\beta$ are both nonzero, then $a_{1,1} \alpha^{-1} x_1 + a_{1,3} \beta^{-1} x_3$ is in the kernel of $H_N(\Delta) \to \overline{H}_N(\Delta, \mu)$. If $\alpha = 0$ or $\beta = 0$, then $\mu$ is supported on a subcomplex with $h$-vector $(1, 2, 1)$. Therefore there is no Gorenstein quotient of $H_N(\Delta)$ with Hilbert function $(1, 3, 1)$. 
\end{example}

If $\Delta$ is a 
simplicial manifold,
then $\dim \overline{H}^d(\Delta) = 1$, and there is a canonical isomorphism $\deg \colon H^d(\Delta) \to K$ given by $\deg(x^F) = 1/[F]$ for each facet $F$.
For an odd-dimensional simplicial manifold, the determinant $D_{d/2}$ of the symmetric bilinear form $\overline{H}^{d/2}(\Delta) \times \overline{H}^{d/2}(\Delta) \to K$ is a well-defined element in $K^{\times}/(K^{\times})^2$.  In \cite[Theorem 1.1]{LNS}, the authors proved that $D_{d/2} = \lambda \prod_{F \text{ facet of }\Delta} [F] \in K^{\times}/(K^{\times})^2$
for some $\lambda \in k^{\times}/(k^{\times})^2$.
Example~\ref{ex:nongeneric} shows that the proof strategy does not extend to doubly Cohen--Macaulay complexes, and the following example shows that the natural generalization of \cite[Theorem 1.1]{LNS} is false. 

\begin{example}
Let $\Delta$ be the $1$-dimensional complex on the right of Figure~\ref{fig:nongeneric}, which is doubly Cohen--Macaulay and has $h$-vector $(1, 2, 2)$.  There is no nonzero cycle  $\mu \in \tilde{H}_1(|\Delta|, k)$ such that the determinant $D_{d/2,\mu}  \in K^{\times}/(K^{\times})^2$ of the symmetric bilinear form on $\overline{H}(\Delta, \mu)$ is divisible by 
$[\{i,j\}]$ for every edge $\{i, j\}$ of $\Delta$. 

Explicitly, a basis of $\tilde{H}_1(|\Delta|, K)$ is given by the classes of the induced subcomplexes with vertex sets $\{1, 2, 4 \}$ and $\{2, 3, 4 \}$, respectively. Call these cycles $\mu_1$ and $\mu_2$. Let $\mu = \alpha \mu_1 + \beta \mu_2$ be a nonzero cycle. Then
\[
D_{d/2,\mu} = \begin{cases}
 \alpha \beta [\{1,2\}][\{1,4\}][\{2,3\}][\{3,4\}] &\textrm{ if } \alpha, \beta \neq 0, \\
\alpha [\{1,2\}][\{1,4\}][\{2,4\}] &\textrm{ if }  \alpha \neq 0, \beta = 0, \\
 	\beta [\{2,3\}][\{2,4\}][\{3,4\}] &\textrm{ if }  \alpha = 0, \beta \neq 0. \\
\end{cases}
\]

\end{example}

\section{Constructions and M-vectors}

In this section, we discuss $M$-vectors and Hilbert functions of modules. We give a strengthening of the bounds in Corollary~\ref{cor:complementaryvector}. 
We also give a construction which shows that every sum of $M$-vectors occurs as the complementary vector of a complex with a convex ear decomposition, completing the proof of Corollary~\ref{cor:complementaryvector}. 

\smallskip

Recall that an $M$-vector is the Hilbert function of a finite dimensional standard graded algebra. 
There are explicit inequalities which characterize $M$-vectors 
that were proved essentially by Macaulay (see \cite[p. 56]{StanleyCombinatoricsCommutative}), which we now recall. 

Given positive integers $b$ and $i$, there is a unique expression
$$b = \binom{n_i}{i} + \binom{n_{i-1}}{i-1} + \dotsb + \binom{n_j}{j}, \quad \text{with} \quad n_i > n_{i-1} > \dotsb > n_j \ge j \ge 1.$$
We define $b^{\langle i \rangle} = \binom{n_i + 1}{i + 1} + \binom{n_{i-1} + 1}{i} + \dotsb + \binom{n_j + 1}{j + 1}$. We set $0^{\langle i \rangle} = 0$. Then we have the following result.
\begin{proposition}\cite{MacaulayPropertiesEnumeration}\label{prop:Mvec}
A sequence of nonnegative integers $(a_0, \dotsc, a_m)$ is an $M$-vector if and only if $a_0 = 1$ and $a_{i+1} \le a_i^{\langle i \rangle}$ for $1 \le i \le m - 1$. 
\end{proposition}

We will need a generalization of Proposition~\ref{prop:Mvec} which classifies the Hilbert function of modules which are generated in degree $0$ over a 
polynomial ring in finitely many variables.
This was proved by Hulett \cite{HulGeneralizationMacaulay} in characteristic $0$ and by Blancafort and Elias \cite{BlancafortElias} in arbitrary characteristic. 

\begin{proposition}\label{prop:Mvecmodule}\cite[Theorem 3.2]{BlancafortElias}
A sequence of nonnegative integers $(a_0, \dotsc, a_m)$ is the Hilbert function of a graded module which is generated in degree $0$ over a 
polynomial ring  with $s$ generators if and only if $a_{i+1} \le q \binom{s + i}{i + 1} + r^{\langle i \rangle}$ for all $0 \le i \le m - 1$, where $q$ and $r$ are the unique nonnegative integers with $a_{i} = q \binom{s + i -1}{i} + r$ and $r < \binom{s + i - 1}{i}$. 
\end{proposition}

We use Proposition~\ref{prop:Mvecmodule} to give a description of the inequalities satisfied by sums of $M$-vectors.

\begin{proposition}\label{prop:summvec}
For a sequence $(a_0, \dotsc, a_m)$ of nonnegative integers, the following are equivalent:
\begin{enumerate}
\item\label{item:1} $(a_0, \ldots, a_m)$ is a sum of $M$-vectors.
\item\label{item:2} Either the sequence is identically $0$, or $a_0 > 0$ and $(1,a_1, \ldots, a_m)$ is an $M$-vector.
\item\label{item:3} $(a_0, \ldots, a_m)$ is the Hilbert function of a graded module which is generated in degree $0$ over a polynomial ring in finitely many variables. 
\end{enumerate}
\end{proposition}

\begin{proof}
\eqref{item:1} implies \eqref{item:3}: A standard graded algebra $A$ is a module over the polynomial ring $\operatorname{Sym} A^1$ which is generated in degree $0$, and so an $M$-vector is the Hilbert function of a module which is generated in degree $0$. We realize a sum of $M$-vectors as the Hilbert function of a module which is generated in degree $0$ by taking the direct sum of these modules. 

\eqref{item:3} implies \eqref{item:2}: If $N$ is a module which is generated in degree $0$ over a polynomial ring in $s$ generators, then it is also a module which is generated in degree $0$ over a polynomial ring in $t$ generators for any $t \ge s$, with the extra generators acting by $0$. For $t$ sufficiently large, the inequalities in Proposition~\ref{prop:Mvecmodule} become the inequalities in \eqref{item:2}.

\eqref{item:2} implies \eqref{item:1}: We can write $(a_0, \dotsc, a_m) = (1, a_1, \dotsc, a_m) + \sum_{i=1}^{a_0 - 1} (1, 0, \dotsc, 0)$ if $a_0 > 0$, and the sequence is the empty sum of $M$-vectors if it is identically zero. 
\end{proof}

\begin{remark}\label{rem:internalzeros}
A sum of $M$-vectors $(a_0,\ldots,a_m)$ has no internal zeros in the sense that if $a_i > 0$ for some $i > 0$, then $a_{i - 1} > 0$. 
\end{remark}

The proof of Corollary~\ref{cor:complementaryvector} shows that the complementary vector of $\Delta$ is the Hilbert function of a graded module which is generated in degree $0$ over $\overline{H}(\Delta)$. Because $\overline{H}(\Delta)$ is a quotient of a polynomial ring with $\overline{h}_1$ generators, Proposition~\ref{prop:Mvecmodule} gives a sharper bound on $\overline{c}(\Delta)$ in terms of $\overline{h}_1$. 

\begin{corollary}\label{cor:modulebound}
Let $\Delta$ be a simplicial complex of dimension $d-1$. 
Let $k$ be a field of characteristic $2$, and assume that $\beta_{d - 1} \neq 0$. 
We have $\overline{c}_{i+1} \le q_i \binom{\overline{h}_1 + i}{i + 1} + r^{\langle i \rangle}_i$ for all 
$0 \le i \le (d - 3)/2$, where $q_i$ and $r_i$ are the unique nonnegative integers with $\overline{c}_{i} = q_i \binom{\overline{h}_1 + i -1}{i} + r_i$ and $r_i < \binom{\overline{h}_1 + i - 1}{i}$. 
\end{corollary}

Because the independence complex of a coloop-free matroid is doubly Cohen--Macaulay \cite[p. 94]{StanleyCombinatoricsCommutative}, Corollary~\ref{cor:modulebound} gives new inequalities for the $h$-vectors of independence complexes of coloop-free matroids. There has been great interest in proving inequalities for $h$-vectors of independence complexes of matroids. For example, the authors of \cite{ADH} and \cite{BST} independently proved that the $h$-vector of the independence complex $(h_0, \dotsc, h_d)$ is log-concave, i.e., $h_i^2 \ge h_{i-1}h_{i+1}$ for $1 \le i \le d-1$. In \cite[Section 6.10]{BrylawksiTutte}, Brylawski conjectures that the $h$-vector of the \emph{broken circuit complex} of a matroid is ultra log-concave, i.e., it remains log-concave after dividing each entry by appropriate binomial coefficients. The independence complex of a matroid is the broken circuit complex of the free coextension of that matroid, so this conjecture also predicts inequalities for independence complexes of matroids. Luis Ferroni pointed out to us that there are small counterexamples to this conjecture. 

\begin{example}\label{ex:ultralog}
There is a matroid of rank $4$ on $6$ elements for which the $h$-vector of the independence complex is $(1, 2, 2, 2, 2)$, which is not ultra log-concave and so gives a counterexample to Brylawski's conjecture. 
\end{example}

We next turn our attention to constructing examples. We will use the following lemma of Chari. 	If $h(\Delta) = (h_0,\ldots,h_d)$ is the $h$-vector of a $(d-1)$-dimensional simplicial complex, then we write $h(\Delta;t) = \sum_{i = 0}^d h_i t^i$.

\begin{lemma}\label{lem:Charicompute} \cite[Lemma~3]{ChariTwoDecompositions} 
Let $\Delta$ be a $(d-1)$-dimensional simplicial complex that admits a convex ear decomposition, i.e.,  it contains a sequence of subcomplexes $\Sigma_1,\ldots,\Sigma_m$, where each $\Sigma_p$ is the boundary of a $d$-dimensional simplicial polytope,  $\Delta = \cup_p \Sigma_p$, and $\Sigma_q \cap (\cup_{p < q} \Sigma_p)$ is a $(d- 1)$-dimensional ball for all $1 < q \le m$. Then
     \[
h(\Delta;t) = h(\Sigma_1;t) + \sum_{q = 2}^m t^d h(B_q;t^{-1}),
\]
where $B_q$ is the closure of $\Sigma_q \smallsetminus (\Sigma_q \cap (\cup_{p < q} \Sigma_p))$.
\end{lemma}

The following proposition completes the proof of Corollary~\ref{cor:complementaryvector}. 

\begin{proposition}\label{prop:construction}
Suppose that a sequence of nonnegative integers $a = (a_0,\ldots,a_{\lfloor (d - 1)/2 \rfloor})$ is a sum of $M$-vectors. 
Then there is a $(d - 1)$-dimensional simplicial complex $\Delta$ that admits a convex ear decomposition  such that $a = \overline{c}(\Delta)$.
\end{proposition}
\begin{proof}

The zero vector is the complementary vector of any $(d - 1)$-dimensional simplicial complex $\Delta$ with $\beta_{d - 1} = 1$, e.g., the boundary of a $d$-dimensional simplicial polytope.
Assume that $a = (a_0,\ldots,a_{\lfloor (d - 1)/2 \rfloor})$ is nonzero. Let $\hat{a} = (1,a_1, \ldots, a_{\lfloor (d - 1)/2 \rfloor})$. Then $a_0 > 0$ and $\hat{a}$ is an $M$-vector by Proposition~\ref{prop:summvec}.
There is a  $(d - 2)$-dimensional simplicial complex $Q$ that is the boundary of a $(d - 1)$-dimensional simplicial polytope with 
$\hat{a} = \overline{g}(Q)$ \cite{BLProofSufficiency}. Recall also that $h(Q;t) = t^{d - 1}h(Q;t^{-1})$. 
Let $S^0 = \{ x, x' \}$ be the $0$-dimensional sphere with two points.
Then the join $Q * S^0$ is a $(d - 1)$-dimensional simplicial sphere with $h(Q * S^0;t) = h(Q;t)(1 + t)$. The subcomplexes $B$ and $B'$ given by the joins $Q * \{ x \}$ and $Q * \{ x' \}$ respectively are $(d- 1)$-dimensional balls with $h$-vectors $h(B;t) = h(B';t) = h(Q;t)$. Let $\Delta'$ be obtained by gluing two copies of $Q * S^0$ along $B$.
Equivalently, $\Delta'$ is obtained by gluing a copy of $B'$ to $Q * S^0$ along $Q$. 
Then $\Delta'$ admits a convex ear decomposition by construction. By Lemma~\ref{lem:Charicompute},
$h(\Delta';t) = h(Q * S^0;t) + t^d h(B';t^{-1}) =  h(Q;t)(1 + 2t)$. 
It follows that $t^d h(\Delta';t^{-1}) - h(\Delta';t) 
= h(Q;t)(1 - t)$, and hence $\overline{c}(\Delta') = \overline{g}(Q) = \hat{a}$. 

Fix a facet $F$ of $\Delta'$. Let $\Delta$ be obtained by gluing $a_0$ copies of $\Delta'$ along the complement of the interior of $F$. Equivalently, $\Delta$ is obtained from $\Delta'$ by successively gluing $a_0 - 1$ copies of $F$ along the boundary of $F$.      Then $\Delta$ admits a convex ear decomposition by construction. Since $h(F;t) = 1$,  Lemma~\ref{lem:Charicompute} implies that $h(\Delta;t) = h(\Delta';t) + (a_0 - 1) t^d$. Hence 
$\overline{c}(\Delta) = \overline{c}(\Delta') + (a_0 - 1, 0,\ldots, 0) = a$. 
\end{proof}

In \cite[Question~7.2]{ATSymmetricDecompositions}, Athanasiadis and Tzanaki asked if the following  inequalities hold:  if $\Delta$ is a $(d - 1)$-dimensional doubly Cohen--Macaulay complex and $h(\Delta) = (h_0(\Delta), \ldots, h_d(\Delta))$, then  
\begin{equation}\label{eq:ATConjecture}
\frac{h_0(\Delta)}{h_d(\Delta)} \le \frac{h_1(\Delta)}{h_{d-1}(\Delta)} \le \cdots \le \frac{h_{d - 1}(\Delta)}{h_1(\Delta)} \le \frac{h_d(\Delta)}{h_0(\Delta)}. 
\end{equation}
As in the proof of Proposition~\ref{prop:construction}, let $Q$ be the boundary of a simplicial $(d - 1)$-dimensional polytope with symmetric $h$-vector $h(Q) = (h_0(Q), \ldots, h_{d-1}(Q))$.  Let $S^0 = \{ x, x' \}$ be the $0$-dimensional sphere. 
Let $\Delta$ be the $(d - 1)$-dimensional doubly Cohen--Macaulay complex obtained by gluing two copies of the suspension $Q * S^0$ of $Q$ along $Q * \{ x \}$. Then we have seen that $h(\Delta;t) = h(Q;t)(1 + 2t)$. Then one computes that 
\eqref{eq:ATConjecture} is equivalent to the condition that 
$h_{i - 1}(Q)h_{i + 1}(Q) \le   h_i(Q)^2$ for all  $i \ge 0$, i.e., $h(Q)$ is a log-concave sequence. The latter condition does not hold in general. 
We give a concrete example below.

\begin{example}\label{ex:AT}
With the notation above,
let $Q$ be the boundary of a simplicial $6$-dimensional polytope with 
$h(Q) = (1,10,13,17,13,10,1)$, which exists by  \cite{BLProofSufficiency}. 
Note that 	$h(Q)$ is not log-concave since $10 \cdot 17 = 170 > 13^2 = 169$. 
Recall that $S^0 = \{ x, x' \}$ is the $0$-dimensional sphere. 
Let $\Delta$ be the $6$-dimensional doubly Cohen--Macaulay complex obtained by gluing two copies of the suspension $Q * S^0$ along $Q * \{ x \}$. Then $h(\Delta) 
= (1,12,33,43,47,36,21,2)$
and $\overline{c}(\Delta) = \overline{g}(Q) = (1,9,3,4)$. We have 
$\frac{h_2(\Delta)}{h_{5}(\Delta)} = \frac{33}{36} > \frac{h_3(\Delta)}{h_{4}(\Delta)} = \frac{43}{47}$, and so $\Delta$ does not satisfy \eqref{eq:ATConjecture}. 
\end{example}

The following example, which is due to Fabrizio Zanello, shows that the conclusion of Corollary~\ref{cor:complementaryvector} need not hold 
for the Hilbert function of finite dimensional  standard graded level algebras satisfying the conclusions of 	Corollary~\ref{cor:topheavy}. 

\begin{example}\label{ex:zanello}

We claim that there is a finite dimensional  standard graded level algebra with Hilbert function
$(1, 19, 19, 20, 19, 19)$. Assuming this claim, the conclusion of Corollary~\ref{cor:topheavy} holds, i.e.,  top heaviness holds and the corresponding $g$-vector $(1,18,0)$ is an $M$-vector. On the other hand, the corresponding complementary vector is $(18,0,1)$, which is not a sum of  $M$-vectors by Remark~\ref{rem:internalzeros}, so the conclusion of Corollary~\ref{cor:complementaryvector} fails.

It remains to establish the claim. We first recall the notion of a trivial extension due to Stanley \cite{StanleyHilbertGradedAlgebras}. 
Let $A$ be a finite dimensional  standard graded level algebra over a field $k$ with Hilbert function $(a_0,\ldots,a_d)$, with $a_d \neq 0$. Let $\omega_A$ be the graded dual of $A$, also known as the canonical module, c.f., the proof of Corollary~\ref{cor:complementaryvector}.
The degree $q$ component of $\omega_A$ is the dual of the vector space $A^{d-q}$. The \emph{trivial extension} is $B = A \oplus \omega_A[1]$. 	 
We regard $B$ as a graded $k$-algebra with multiplication defined by $(a,f) \cdot (a',f') = (a a', a \cdot f' + a' \cdot f)$ for all $a,a' \in A$ and $f,f' \in \omega_A[1]$. 	The hypothesis that $A$ is level is equivalent to the condition that $\omega_A$ is generated in degree $0$ as an $A$-module, which implies that $B$ is standard graded. 
Moreover, multiplication induces a nondegenerate pairing $B^q \times B^{d + 1 - q} \to B^{d + 1} = \omega_A^{d}$, and hence $B$ is Gorenstein.  In particular, $B$ is level. Note that $B$ has Hilbert function 
$(a_0,\ldots,a_d,0) + (0,a_d,\ldots,a_0)$.

For a positive integer $m$, the \emph{truncation} of a standard graded algebra to degree  $m$ is the algebra obtained by quotienting by the ideal generated by the elements of degree at least $m+1$. 
The truncation of a standard graded level algebra to any fixed degree is still a standard graded level algebra. 

Let $A$ be the truncation of  $k[x,y,z]/(yz^3,z^4)$ to degree $5$. Then $A$ is a standard graded level algebra with Hilbert function 
$(1, 3, 6, 10, 13, 16)$. Let $B$ be the trivial extension of $A$ with Hilbert function $(1, 3, 6, 10, 13, 16, 0) + (0, 16, 13, 10, 6, 3, 1) = (1, 19, 19, 20, 19, 19, 1)$. Then the truncation of $B$ to degree $5$  is a standard graded level algebra with Hilbert function
$(1, 19, 19, 20, 19, 19)$.

\end{example}

We now give two examples of level algebras where the conclusion of Corollary~\ref{cor:gelement} holds, but the conclusion  of Theorem~\ref{thm:gorensteinhilbert} may or may not hold. These algebras are level monomial algebras, and Hausel showed that the conclusion of Corollary~\ref{cor:gelement} holds for any level monomial algebra over a field of characteristic $0$ \cite[Theorem 6.3]{HauselQuaternionicMatroids}. See Section~\ref{ss:pure} for further discussion. 

\begin{example}\label{ex:monomial}

Let $k$ be a field and let 
$A = k[x,y,z]/(x^3, x^2y, x^2z,yz,y^2,z^2)$. This is a standard graded level algebra with Hilbert function $(1,3,3)$ and degree $2$ graded piece  $A^2 = k x^2 \oplus k xy \oplus k xz$.
If $\ell = x + y + z$, then 
multiplication by $\ell^2 = x^2 + 2xy + 2 xz$ induces an injection from $A^0$ to $A^2$. However, we claim that $A$ does not admit a Gorenstein quotient with Hilbert function $(1,3,1)$. 

To establish the claim, let  $\mu \colon A^2 \to k$ be a nonzero linear map, and assume that the corresponding Gorenstein quotient $A(\mu)$ has Hilbert function $(1,3,1)$. Then the projection map $\pi \colon A \to A(\mu)$ is an isomorphism in degree $1$. On the other hand, $\mu(xy)z - \mu(xz)y \in \ker(\pi)$, 
and hence $\mu(xy) = \mu(xz) = 0$. Then $y$ and $z$ lie in  the kernel of $\pi$, a contradiction.

\end{example}

\begin{example}\label{ex:monomial2}

Let $k$ be a field and let $A = k[x,y,z]/(x^3,x^2y,xy^2,xz,y^2,z^2)$. 
This is a standard graded level algebra with Hilbert function $(1,3,3)$ and degree $2$ graded piece  $A^2 = k x^2 \oplus k xy \oplus k yz$. 	Let $\mu \colon A^2 \to k$ be a linear map such that $\mu(x^2),\mu(xy),$ and $\mu(yz)$ are all nonzero.
We claim that the corresponding Gorenstein quotient has Hilbert function $(1,3,1)$.

To prove the claim, suppose that 
$ax + by + cz$ lies in the kernel of the map to the Gorenstein quotient, 
for some $a, b$, and $c$ in $k$. Then 
$\mu(z \cdot (ax + by + cz)) = b \mu(yz) = 0$, so $b = 0$. Also,
$\mu(x \cdot (ax + by + cz)) = a \mu(x^2) + b \mu(xy) = 0$, so $a = 0$. Finally,  
$\mu(y \cdot (ax + by + cz)) = a \mu(xy) + c \mu(yz) = 0$, so $c = 0$. 
\end{example}

\section{Generalizations and further questions}

In this section, we discuss various generalizations of our results, as well as questions for future research. 

\subsection{Arbitrary characteristic}\label{ss:positivechar}

Fix a field $k$.  In this section, we do not assume that $k$ has characteristic $2$. 
We let 
$\beta_{k,q} = \dim \tilde{H}_q(|\Delta|;k)$, and we assume that $\beta_{k,d - 1} \neq 0$. 
Let $\overline{H}_k(\Delta)$ be the level quotient of the generic artinian reduction of the Stanley--Reisner ring of $\Delta$ over $k$, and similarly define $\overline{H}_k(\Delta, \mu)$ for $\mu \in \tilde{H}_{d-1}(|\Delta|, k)$.
Define the 	$\overline{h}_k$-vector $(\overline{h}_{k,0}, \ldots, \overline{h}_{k,d})$ to be the Hilbert function of $\overline{H}_k(\Delta)$, i.e., $\overline{h}_{k,q} := \dim \overline{H}_k^q(\Delta)$ for $0 \le q \le d$.  
The proof of Proposition~\ref{prop:kindep} implies that the $\overline{h}_k$-vector only depends on the characteristic of $k$.

The differential operators technique has been extended to the case when $k$ has any positive characteristic \cite{KLS,moment}, and this allows one to prove the following analogue of Theorem~\ref{thm:gorensteinhilbert} in other positive characteristics.

\begin{theorem}\label{thm:gorensteinhilbertcharp}
Let $\Delta$ be a simplicial complex of dimension $d-1$. Let $k$ be a field of characteristic $p > 0$, and assume that $\beta_{k,d - 1} \neq 0$. 
Then there is a map $\mu \colon \overline{H}^d_k(\Delta) \to K$ such that, for each $q \le d/p$, $\dim \overline{H}^q_k(\Delta, \mu) = \overline{h}_{k,q}$ and multiplication by $\ell^{d - pq}$ induces an injection from $\overline{H}^q_k(\Delta, \mu)$ to $\overline{H}^{d - (p-1)q}_k(\Delta, \mu)$.  If $k$ is infinite, then $\mu$ can be chosen to be defined over $k$. 
\end{theorem}

One can then use the techniques of the paper to deduce properties of complementary vectors, but only for the first $\lfloor d/p \rfloor + 1$ entries.

On the other hand, the anisotropy of the Hodge--Riemann forms on $\overline{H}(\Delta, \mu)$ is not known to hold if $k$ does not have characteristic $2$. 
If the anisotropy of Hodge--Riemann forms were known over fields of characteristic $p$, then the techniques of this paper would show that
the bound $q \le d/p$ in Theorem~\ref{thm:gorensteinhilbertcharp} could be replaced by $q \le d/2$. This would allow one to extend our results on complementary vectors. For 
example, we would conclude that the complementary vectors of complexes which are doubly Cohen--Macaulay in characteristic $p$ are sums of $M$-vectors. 

We outline the proof of Theorem~\ref{thm:gorensteinhilbertcharp}. We use the following generalization of  Proposition~\ref{prop:anistropycycle}. It is a minor modification of \cite[Theorem~1.3]{KLS}, obtained just as in the proof of Proposition~\ref{prop:anistropycycle}.

\begin{proposition}\label{prop:anistropycyclecharp}
Let $k$ be a field of characteristic $p > 0$. 
Let $\mu \colon \overline{H}^d_k(\Delta) \to K$ be a nonzero map which is defined over $\mathbb{F}_p$, and let $\overline{H}_k(\Delta, \mu)$ be the corresponding Gorenstein quotient. Let $\ell = x_1 + \dotsb + x_n \in \overline{H}^1_k(\Delta, \mu)$. 
 For each $q \le d/p$ and any nonzero $g \in \overline{H}^q_k(\Delta, \mu)$, 
$\ell^{d - pq} \cdot g^p$ is nonzero.
\end{proposition}

\begin{proof}[Proof of Theorem~\ref{thm:gorensteinhilbertcharp}]
The proof of Theorem~\ref{thm:gorensteinhilbert} extends to a proof of Theorem~\ref{thm:gorensteinhilbertcharp}, with two modifications.

Firstly,  for some $q \le d/p$, we let $V = \overline{H}^q_k(\Delta)$ and let $W$ be the $K$-vector space of $(p - 1)$-multilinear maps from $V$ to $K$, i.e., the tensor product of $\operatorname{Hom}(\overline{H}^q_k(\Delta), K)$ with itself $p - 1$ times. For $\nu \in \operatorname{Hom}(\overline{H}^d_k(\Delta), K)$, define 
$\phi_{\nu} \colon V \to W$ to be the map given by $\phi_{\nu}(x) = ((y_1,\ldots,y_{p - 1}) \mapsto \nu(x \cdot y_1 \cdots y_{p - 1} \cdot \ell^{d - pq}))$. Multiplication by $\ell^{d - pq}$ induces an injection from $\overline{H}^q_k(\Delta, \nu)$ to $\overline{H}^{d - (p-1)q}_k(\Delta, \nu)$ if and only if   
$\ker(\phi_{\nu})$, which is $\{x \in \overline{H}^q_k(\Delta) : \nu(x \cdot  y_1 \cdots y_{p - 1} \cdot \ell^{d - pq}) = 0 \text{ for all }  y_1, \ldots, y_{p - 1} \in \overline{H}^q_k(\Delta)\}$,
coincides with the kernel of the map $\overline{H}^q_k(\Delta) \to \overline{H}^q_k(\Delta, \nu)$. 

Secondly, the proof of \eqref{eq:phi} is modified as follows: let $x$ be a class in $\ker(\phi_{\mu})$ such that there is some $y \in \overline{H}^q_k(\Delta)$ with $\phi_{\mu_{t+1}}(x) = \phi_{\mu}(y)$.
For any $z_1,\ldots,z_{p - 1} \in \overline{H}^q_k(\Delta)$, we have
$$\mu_{t+1}(x \cdot z_1 \cdots z_{p - 1} \cdot \ell^{d - pq}) = \mu(y \cdot z_1 \cdots z_{p - 1} \cdot \ell^{d - pq}).$$
Set $z_1 = \cdots = z_{p - 1} = x$ in the above equation. As $x \in \ker(\phi_{\mu})$, we have $\mu(y \cdot x^{p - 1} \cdot \ell^{d - pq}) = 0$. Therefore $\mu_{t+1}(x^p \cdot \ell^{d - pq}) = 0$. Proposition~\ref{prop:anistropycyclecharp} then implies that $x$ lies in the kernel of $\overline{H}^q_k(\Delta) \to \overline{H}^q_k(\Delta, \mu_{t+1})$, which is the kernel of $\phi_{\mu_{t+1}}$. 
\end{proof}

In some cases, we can deduce an analogue of Theorem~\ref{thm:gorensteinhilbert} when $k$ has characteristic $0$. The hypothesis of the following theorem holds, for example, when $\Delta$ is doubly Cohen--Macaulay over a field of characteristic $2$, which implies that it is doubly Cohen--Macaulay over $k$.

\begin{theorem}\label{thm:char0}
Let $\Delta$ be a simplicial complex of dimension $d-1$. Let $k$ be a field of characteristic $0$, and assume that $\beta_{k,d - 1} \neq 0$ and that $\dim \overline{H}^q_{k}(\Delta) = \dim \overline{H}^q_{{\mathbb{F}_2}}(\Delta)$ for all $q$. 
Then there is a map $\mu \colon \overline{H}^d_k(\Delta) \to K$ which is defined over $k$ such that $\dim \overline{H}^q_k(\Delta, \mu) = \overline{h}_{k,q}$ for each $q \le d/2$. 
\end{theorem}

\begin{proof}
Choose a basis $\mu_1, \dotsc, \mu_t$ for $\tilde{H}_{d-1}(|\Delta|, \mathbb{Q})$.  Then for $a_1, \dotsc, a_t \in \mathbb{Q}$, the dimension of $\overline{H}^{q}_{\mathbb{Q}}(\Delta, \sum a_i \mu_i)$ is the rank of a certain matrix computing the pairing between monomials of degree $q$ and $d - q$. The entries of this matrix are polynomials in $a_1, \dotsc, a_t$ with rational coefficients, and so if there is some field $k$ of characteristic $0$ for which the conclusion of the theorem holds, then it holds for every field of characteristic $0$. 

The assumption that $\dim \overline{H}^d_{k}(\Delta) = \dim \overline{H}^d_{{\mathbb{F}_2}}(\Delta)$ and the universal coefficient theorem for homology imply that $\dim \tilde{H}_{d-1}(|\Delta|, \mathbb{Q}) = \dim \tilde{H}_{d-1}(|\Delta|, \mathbb{F}_2)$. By Theorem~\ref{thm:gorensteinhilbert}, we can find a cycle $\overline{\mu} \in \tilde{H}_{d-1}(|\Delta|, \bar{\mathbb{F}}_2)$ for which the Hilbert function of the corresponding Gorenstein quotient agrees with $\dim \overline{H}^q_k(\Delta)$ for $q \le d/2$. 
Here $\bar{\mathbb{F}}_2$ denotes the algebraic closure of $\mathbb{F}_2$. 
As $\bar{\mathbb{F}}_2$ is a union of finite fields, $\overline{\mu}$ is defined over $\mathbb{F}_{2^n}$ for some $n$. 

Let $k$ be a number field whose ring of integers $\mathcal{O}_k$ has $\mathbb{F}_{2^n}$ as a residue field. Then there is a cycle $\mu \in \tilde{H}_{d-1}(|\Delta|, \mathcal{O}_k)$ which specializes to $\overline{\mu}$. This implies that the rank of the matrix which computes the pairing between monomials of degree $q$ and $d-q$ using $\mu$ is $\dim \overline{H}^q_k(\Delta)$,  as the matrix has this rank after reducing to $\mathbb{F}_{2^n}$. That is also the rank of this matrix over $k$, so $\dim \overline{H}^q_k(\Delta, \mu) = \dim \overline{H}^q_{\mathbb{F}_{2^n}}(\Delta, \mu) = \dim \overline{H}^q_k(\Delta)$, as desired.
\end{proof}

\begin{remark}
	In the statement of Theorem~\ref{thm:char0}, a standard argument can be used to show that $\mu$ may be chosen so that $\overline{H}_k(\Delta, \mu)$ satisfies the strong Lefschetz property, see, e.g., \cite[Proof of Theorem 1.4]{LNS}. 
\end{remark}

\subsection{Balanced complexes}\label{ss:balanced}

In \cite[Theorem~7.6]{Oba}, Oba generalized Proposition~\ref{prop:anistropycycle} to the case of balanced complexes. We explain how this result can be used to extend the results of the paper to the balanced complex setting.

Let $\Delta$ be a simplicial complex of dimension $d-1$ with vertex set $V = \{ 1,\ldots, n \}$. Let $\textbf{a} = (a_1,\ldots,a_m)$ be a sequence of positive integers with $\sum_q a_q = d$. Recall that an $m$-coloring of the vertices of $\Delta$ is a function  $\kappa \colon V \to \{ 1,\ldots, m\}$ such that $\kappa(v) \neq \kappa(w)$ for any edge $\{ v,w \}$ in $\Delta$. 
An \emph{$\textbf{a}$-balanced complex} is a pair  $(\Delta, \kappa)$, where $\kappa \colon V \to \{ 1,\ldots, m\}$ is a coloring of the vertices of $\Delta$, and for each face $F$ of $\Delta$, $|F \cap \kappa^{-1}(q)| \le a_q$ for $1 \le q \le m$.  

Define a function $\psi \colon \{ 1,\ldots, d \} \to \{ 1,\ldots, m \}$ by setting
$\psi(i) = q $ for $i \in \{\sum_{l < q} a_l + 1, \ldots, \sum_{l < q} a_l + a_q\}$. 
Let $k$ be a field of characteristic $2$, and consider a purely transcendental field extension $K = k(a_{i,j})$, where the indices $(i,j)$ run over all 
$1 \le i \le d$ and $1 \le j \le n$ such that $\psi(i) = \kappa(j)$.
Recall that $K[\Delta]$ denotes the Stanley--Reisner ring of $\Delta$. Let $\theta_i = \sum_{\psi(i) = \kappa(j)} a_{i,j} x_j$ for $i \in \{1, \dotsc, d\}$, and define $H(\Delta, \kappa) \coloneqq K[\Delta]/(\theta_1, \dotsc, \theta_d)$. 
Then $H(\Delta, \kappa) = \oplus_{\textbf{b} \le \textbf{a}} H^{\textbf{b}}(\Delta, \kappa)$ naturally inherits an $\N^m$-grading that refines the usual $\N$-grading. Here $\textbf{b} = (b_1, \ldots , b_m) \in \N^m$, and we write $\textbf{b}  \le \textbf{a}$ if  $b_q \le a_q$ for $1 \le q \le m$. 
We assume that $H^d(\Delta, \kappa) = H^{\textbf{a}}(\Delta, \kappa)$ is nonzero. 
There is a natural identification of $H^{\textbf{a}}(\Delta, \kappa)$ with $\tilde{H}^{d-1}(|\Delta|, K)$. 
We may consider the $\N^m$-graded level quotient $\overline{H}(\Delta, \kappa)$, and, for each nonzero map $\mu \colon \overline{H}^d(\Delta, \kappa) \to K$, the $\N^m$-graded Gorenstein quotient  $\overline{H}(\Delta, \kappa, \mu)$. 
Let $\ell_q = \sum_{\kappa(j) = q} x_j \in \overline{H}^{e_q}(\Delta, \kappa)$ for $1 \le q \le m$, where $e_q \in \N^m$ denotes the $q$th standard basis vector. Define 
$\boldsymbol{\ell}^{\textbf{b}}:= \prod_q \ell_q^{b_q} \in \overline{H}^{\textbf{b}}(\Delta, \kappa)$. Then the following result is essentially  the same as \cite[Theorem 7.6]{Oba}.

\begin{proposition}\label{thm:SLPOba}
Let $(\Delta, \kappa)$ be a balanced complex with $\dim \Delta = d - 1$.  Let $k$ be a field of characteristic $2$, and let $\mu \colon \overline{H}^d(\Delta, \kappa) \to K$ be a nonzero map which is defined over $\mathbb{F}_2$. Then for any $\textbf{b} \in \N^m$ with $2\textbf{b} \le \textbf{a}$, the symmetric bilinear form $\overline{H}^{\textbf{b}}(\Delta, \kappa, \mu) \times \overline{H}^{\textbf{b}}(\Delta, \kappa, \mu) \to K$ given by 
$(x, y) \mapsto \mu(x \cdot y \cdot \boldsymbol{\ell}^{\textbf{a} - 2\textbf{b}})$ is anisotropic. 
\end{proposition}

Define the $\overline{h}$-vector of $(\Delta,\kappa)$ by 
$\overline{h}_{\textbf{b}} := \dim \overline{H}^{\textbf{b}}(\Delta, \kappa)$ for $\textbf{b} \le \textbf{a}$. As in Proposition~\ref{prop:kindep}, this is independent of the choice of characteristic $2$ field $k$.  
A minor modification of the 
proof of Theorem~\ref{thm:gorensteinhilbert} implies the following result. 

\begin{theorem}\label{thm:gorensteinhilbertbalanced}
Let $(\Delta, \kappa)$ be a balanced complex with $\dim \Delta = d - 1$. 
Let $k$ be a field of characteristic $2$, and assume that $\beta_{d - 1} \neq 0$. 
Then there is a map $\mu \colon \overline{H}^d(\Delta, \kappa) \to K$ such that, for each $\textbf{b} \in \N^m$ with $2\textbf{b} \le \textbf{a}$, $\dim \overline{H}^{\textbf{b}}(\Delta, \kappa,\mu) = \overline{h}_{\textbf{b}}$ and multiplication by $\boldsymbol{\ell}^{\textbf{a} - 2\textbf{b}}$ induces an isomorphism from $\overline{H}^{\textbf{b}}(\Delta, \kappa,\mu)$ to $\overline{H}^{\textbf{a} - \textbf{b}}(\Delta, \kappa,\mu)$. 
If $k$ is infinite, then $\mu$ can be chosen to be defined over $k$.
\end{theorem}

Define the complementary vector of $(\Delta,\kappa)$ by 
$\overline{c}_{\textbf{b}} := \overline{h}_{\textbf{a} - \textbf{b}}  - \overline{h}_{\textbf{b}} $ for $\textbf{b}$ with $2\textbf{b} < \textbf{a}$. 
We now proceed as in the
proof of Corollary~\ref{cor:complementaryvector}.  
Theorem~\ref{thm:gorensteinhilbertbalanced} implies there  exists a linear map
$\mu \colon \overline{H}^d(\Delta,\kappa) \to K$ 
such that $\dim \overline{H}^{\textbf{b}}(\Delta, \kappa,\mu) = \overline{h}_{\textbf{b}}$ for each $\textbf{b} \in \N^m$ with $2\textbf{b} \le \textbf{a}$. 
Let $M$ be the kernel of the quotient map $\overline{H}(\Delta, \kappa) \to \overline{H}(\Delta, \kappa, \mu)$, and let $M^*$ be the graded dual of $M$. Both $M$ and $M^*$ admit  $\N^m$-gradings that are compatible with their perfect pairing. 
The Hilbert function of $M^*$ (with respect to the $\N^m$-grading) is precisely the complementary vector of $(\Delta,\kappa)$. 
Using the standard $\N$-grading, the socle of $M$ is concentrated in degree $d$, and it follows that $M^*$ is generated in degree $0$. 
We conclude that we have the following analogue of Corollary~\ref{cor:complementaryvector}. 

\begin{corollary}\label{cor:complementaryvectorbalanced}
Let $(\Delta, \kappa)$ be a balanced complex with $\dim \Delta = d - 1$.
Let $k$ be a field of characteristic $2$, and assume that $\beta_{d - 1} \neq 0$. 
Then the complementary vector of $(\Delta, \kappa)$ is the Hilbert function of an $\N^m$-graded module over an $\N^m$-graded polynomial ring
 in finitely many variables which is generated as a module in degree $0$. 
\end{corollary}

It would be interesting to further explore the numerical consequences of this result.

\subsection{Pure $O$-sequences}\label{ss:pure}

Recall that an $M$-vector (or $O$-sequence) is the Hilbert function of a finite dimensional standard graded $K$-algebra. In fact, it follows from Macaulay's work \cite{MacaulayPropertiesEnumeration} that an $O$-sequence may alternatively be defined as the Hilbert function of a finite dimensional monomial algebra, i.e., the quotient of a polynomial ring by a monomial ideal. 
A \emph{pure $O$-sequence} is the Hilbert function of a finite dimensional monomial algebra that is level, i.e., the socle is concentrated in top degree. A long-standing conjecture of Stanley asserts that the $h$-vector of the independence complex of a matroid is a pure $O$-sequence \cite{StanleyCohenMacaulayComplexes}. 

Working over $k = \C$, Hausel \cite[Theorem~6.3]{HauselQuaternionicMatroids} proved that finite dimensional level monomial algebras satisfy a natural analogue of Corollary~\ref{cor:gelement}, and so pure $O$-sequences satisfy 
the inequalities in Corollary~\ref{cor:topheavy}. Given a pure $O$-sequence $(a_0,\ldots,a_d)$ with $a_d \neq 0$, define the corresponding complementary vector as $(c_0, \ldots, c_{\lfloor (d - 1)/2 \rfloor})$, where $c_q = a_{d - q} - a_{q}$ for each $q \le (d - 1)/2$. Then Hausel's theorem implies that the complementary vector  has nonnegative entries. We prove the following analogue of part of Corollary~\ref{cor:complementaryvector}. 

\begin{theorem}\label{thm:pure}
The complementary vector of a pure $O$-sequence is a sum of $M$-vectors.
\end{theorem}

Note that Example~\ref{ex:monomial} and Example~\ref{ex:monomial2} show that the analogue of Theorem~\ref{thm:gorensteinhilbert} may or may not hold, i.e., a finite dimensional level monomial algebra may or may not admit a Gorenstein quotient whose Hilbert function is the same up to half of the socle degree.

We prove Theorem~\ref{thm:pure} by reducing it to a theorem of Murai. Let $k$ be a field and let $M$ be the set of monomials in $k[x_1,\ldots,x_n]$. For $u,v \in M$, write $u \mid v$ if $u$ divides $v$. A nonempty finite subset  $S \subset M$ is a \emph{multicomplex} if whenever $u \mid v$ and $v \in S$, then $u \in S$. Equivalently, the elements of $S$ form a basis of a quotient of $k[x_1,\ldots,x_n]$ by a monomial ideal. 
Write $h_q(S)$ for the number of elements of $S$ of degree $q$. 

\begin{example}\label{ex:pureGorenstein}
If $\hat{m} \in M$ is a monomial of degree $d$, then $S = \{ u \in M : u \mid \hat{m} \}$ is the basis of a Gorenstein monomial algebra with socle degree $d$ and $h_q(S) = h_{d - q}(S)$ for all $q$. Moreover, the function 
$u \mapsto \hat{m}/u$ defines an involution on $S$ sending elements of degree $q$ to elements of degree $d - q$. 
\end{example}

\begin{theorem}\cite[Theorem~1.14]{MurSpheres}\label{thm:Murai}
	Fix a monomial
	$\hat{m} \in M$ of degree $d$ and  let $S$ be a multicomplex that is a proper subcomplex of  $\{ u \in M : u \mid \hat{m} \}$. 
	Consider the sequence  $g(S) = (g_0(S),\ldots,g_{\lfloor (d - 1)/2 \rfloor}(S))$, where  $g_q(S) = h_q(S) - h_{d - q}(S)$. Then $g(S)$ is the $g$-vector of a simplicial sphere and, in particular, is an $M$-vector. 
\end{theorem}

 \begin{proof}[Proof of Theorem~\ref{thm:pure}]
 	With the notation above, let $m_1,\ldots,m_s \in M$ be a sequence of distinct monomials of degree $d$ for some positive integer $s$ and let $A = \{ u \in M : u \mid m_q \textrm{ for some } 1 \le q \le s\}$. Then any pure $O$-sequence has the form $(h_0(A),\ldots,h_d(A))$ for some such multicomplex $A$. Consider the complementary vector
 	$c(A) = (c_0(A),\ldots,c_{\lfloor (d - 1)/2 \rfloor}(A))$, where $c_q(A) = h_{d - q}(A) - h_q(A)$. 
 	We argue by induction on $s$. When $s = 1$ the complementary vector is identically zero by Example~\ref{ex:pureGorenstein} and the result holds. Assume that $s > 1$ and consider the multicomplexes
 	$$A' = \{ u \in M : u \mid m_q \textrm{ for some }  1 \le q < s \} \subset A,$$ 
 	$$A_s =  \{ u \in M : u \mid m_s \} \subset A.$$
 	Consider the multicomplex $S = A' \cap A_s$. Since $m_s \in A_s \smallsetminus A'$, $S \neq A_s$. 
 	Consider the sequence  $g(S) = (g_0(S),\ldots,g_{\lfloor (d - 1)/2 \rfloor}(S))$, where  $g_q(S) = h_q(S) - h_{d - q}(S)$. 
 	By Theorem~\ref{thm:Murai}, $g(S)$ is an $M$-vector. 
 	We claim that for $0 \le q \le (d - 1)/2$, 
 	\[
 	c_q(A) = c_q(A') + g_q(S). 
 	\]
 	Assuming this claim, the result follows immediately by induction. It remains to prove the claim.
 	Observe that  $A_s \smallsetminus S = A \smallsetminus A'$. 
 	Hence 
 	\begin{align*}
 		c_q(A) - c_q(A') &= h_{d - q}(A \smallsetminus A') - h_q(A \smallsetminus A') \\
 		&= h_{d - q}(A_s \smallsetminus S) - h_q(A_s \smallsetminus S). 
 	\end{align*}
 	Recall from Example~\ref{ex:pureGorenstein} that $A_s$ admits an involution $\phi: A_s \to A_s$, $\phi(u) = m_s/u$, taking elements of degree $q$ to elements of degree $d - q$. 
 	Write $A_s \smallsetminus S$ as the disjoint union of  
 	$B = \{ u \in A_s \smallsetminus S : \phi(u) \in S \}$ and 
 	$B' = \{ u \in A_s \smallsetminus S : \phi(u) \notin S \}$. Since $B'$ is closed under $\phi$, we have $h(B')_q = h(B')_{d - q}$ for all $q$. 
 	Hence
 	\begin{align*}
 		 	h_{d - q}(A_s \smallsetminus S) - h_q(A_s \smallsetminus S) &= h_{d - q}(B) - h_q(B) 
 		 	\\
 		 	&= h_q(\phi(B)) - h_{d - q}(\phi(B)).
 	\end{align*}
	We have $\phi(B) = \{ u \in S : \phi(u) \notin S \}$. Observe that $S$ is the disjoint union of $\phi(B)$ and $C = \{ u \in S : \phi(u) \in S \}$. Since $C$ is closed under $\phi$, we have $h(C)_q = h(C)_{d - q}$ for all $q$. We deduce that 
	\begin{align*}
		h_q(\phi(B)) - h_{d - q}(\phi(B)) &= h_q(S) - h_{d - q}(S)  = g_q(S).
	\end{align*} 
 	This completes the proof of the claim and hence the theorem. 
 \end{proof}

\subsection*{Conflicts of interest:} none.
\subsection*{Financial support:} This work was mostly conducted while the authors were at the Institute for Advanced Study, where the second author received support from the Charles Simonyi endowment.

\bibliography{cycles}
\bibliographystyle{amsalpha}

\end{document}